\documentclass[12pt]{amsart}
\usepackage{graphicx}
\usepackage{latexsym}
\usepackage{fancyhdr}
\usepackage{amsmath, amssymb}
\usepackage[all]{xy}
\usepackage{pdflscape}
\usepackage{longtable}
\usepackage{rotating}
\usepackage{verbatim}
\usepackage{hyperref}
\usepackage{subfigure}
\usepackage{mathrsfs}
\usepackage{tensor}
\usepackage{mdwlist}
\usepackage{etoolbox}
\usepackage{todonotes}
\usepackage{mathtools}

\theoremstyle{plain}
\newtheorem*{theorem*}{Theorem}
\newtheorem{theorem}{Theorem}[section]
\newtheorem*{lemma*}{Lemma}
\newtheorem{lemma}[theorem]{Lemma}
\newtheorem*{proposition*}{Proposition}
\newtheorem{proposition}[theorem]{Proposition}
\newtheorem*{corollary*}{Corollary}
\newtheorem{corollary}[theorem]{Corollary}
\newtheorem*{claim*}{Claim}
\theoremstyle{definition}
\newtheorem*{assumption*}{Assumption}

\newtheorem*{definition*}{Definition}
\newtheorem{definition}[theorem]{Definition}
\newtheorem*{convention*}{Convention}
\newtheorem*{example*}{Example}

\newtheorem*{algorithm*}{Algorithm}
\newtheorem*{remark*}{Remark}

\numberwithin{equation}{section}

\def\al{\alpha}
\def\be{\beta}

\def\de{\delta}

\def\la{\lambda}

\def\si{\sigma}

\def\vh{\varphi}

\def\ps{\psi}

\def\De{\Delta}

\def\C{\mathbb{C}}

\def\N{\mathbb{N}}

\def\R{\mathbb{R}}

\def\p{\partial}

\def\<{\langle}
\def\>{\rangle}
\renewcommand{\o}{\circ}
\newcommand{\ti}{\tilde}

\let\on=\operatorname
\newcommand{\sr}[1]%
{\ifmmode{}^\dagger\else${}^\dagger$\fi\ifvmode
\vbox to 0pt{\vss
 \hbox to 0pt{\hskip\hsize\hskip1em
 \vbox{\hsize3cm\raggedright\pretolerance10000
 \noindent #1\hfill}\hss}\vss}\else
 \vadjust{\vbox to0pt{\vss%
 \hbox to 0pt{\hskip\hsize\hskip1em%
 \vbox{\hsize3cm\raggedright\pretolerance10000%
 \noindent #1\hfill}\hss}\vss}}\fi%
}

\newcommand{\tC}[2]{\tensor[^{#1}]{C}{^{#2}}}

\title[A new proof of Bronshtein's theorem]
{A new proof of Bronshtein's theorem}

\author[Adam Parusi\'nski and  Armin Rainer]
{Adam Parusi\'nski and Armin Rainer}

\address {Adam Parusi\'nski: Univ. Nice Sophia Antipolis, CNRS,  LJAD, UMR 7351, 06108 Nice, France}

\email{adam.parusinski@unice.fr}

\address{Armin Rainer: Fakult\"at f\"ur Mathematik, Universit\"at Wien, 
Oskar-Morgenstern-Platz~1, A-1090 Wien, Austria}

\email{armin.rainer@univie.ac.at}

\begin{document}

%-----------------------------------------------------------------------------------------------------------------------
\begin{abstract}
  We give a new self-contained proof of Bronshtein's theorem, that  
  any continuous root of a $C^{n-1,1}$-family of monic 
  hyperbolic polynomials of degree $n$ is locally Lipschitz, and 
  obtain explicit bounds for the Lipschitz constant of the root in 
  terms of the coefficients.
  As a by-product we reprove the recent result of Colombini, Orr\'u, 
  and Pernazza, that a $C^n$-curve of hyperbolic polynomials of 
  degree $n$ admits a $C^1$-system of its roots.
\end{abstract}
%-----------------------------------------------------------------------------------------------------------------------

\thanks{Supported by the Austrian Science Fund (FWF), Grants P~22218-N13 and P~26735-N25, and by ANR project STAAVF (ANR-2011 BS01 009).}
%\dedicatory{dedicatory}
\date{May 7, 2014}

\maketitle

%-----------------------------------------------------------------------------------------------------------------------
\section{Introduction}
%-----------------------------------------------------------------------------------------------------------------------

Choosing regular roots of polynomials whose coefficients depend on parameters is a classical much studied problem 
with important connections to various fields such as algebraic geometry, partial differential equations, and 
perturbation theory. 

This problem is of special interest for \emph{hyperbolic} polynomials whose roots are all real. 
Probably the first result in this direction was obtained by Glaeser \cite{Glaeser63R} who studied the square root of a 
nonnegative smooth function. The most important and most difficult result in this field is Bronshtein's theorem 
\cite{Bronshtein79}: any continuous root of a $C^{p-1,1}$-curve of monic 
hyperbolic polynomials, where $p$ is the maximal multiplicity of the roots, 
is locally Lipschitz with uniform Lipschitz constants; cf.~Theorem~\ref{main}. 
A multiparameter version follows immediately; see Theorem~\ref{main2}. 
A different proof was later given by Wakabayashi \cite{Wakabayashi86} who actually proved a more general H\"older 
version; for a refinement of Bronshtein's method in order to show this generalization see Tarama~\cite{Tarama06}. 
Kurdyka and Paunescu \cite{KurdykaPaunescu08} used resolution of singularities to show that the roots of a hyperbolic 
polynomial whose coefficients are real analytic functions in several variables admit a parameterization which is 
locally Lipschitz; in one variable we have Rellich's classical theorem \cite{Rellich37I} that the roots may be 
parameterized by real analytic functions.    

A $C^p$-curve of monic hyperbolic polynomials with at most $p$-fold roots admits a differentiable system of 
its roots. Using Bronshtein's theorem, Mandai \cite{Mandai85} showed that the roots can be chosen $C^1$ 
if the coefficients are $C^{2p}$, and Kriegl, Losik, and Michor \cite{KLM04} found twice differentiable roots 
provided that the coefficients are $C^{3p}$. 
%Using Bronshtein's theorem, it was shown that the roots can be parameterized by $C^1$, 
%respectively twice differentiable, functions if the coefficients are $C^{2p}$, respectively $C^{3p}$, 
%by Mandai \cite{Mandai85}, respectively by Kriegl, Losik, and Michor \cite{KLM04}. 
Recently, Colombini, Orr\'u and Pernazza \cite{ColombiniOrruPernazza12} 
proved that $C^p$ (resp.\ $C^{2p}$) coefficients suffice for $C^1$ (resp.\ twice differentiable) roots
and that this statement is best possible.

In this paper we present a new proof of Bronshtein's theorem. 
Our proof is simple and elementary.  The main tool is the splitting principle, a criterion that allows to factorize 
polynomials under elementary assumptions.  The coefficients of the factors can be expressed in a simple way 
in terms of the coefficients of the original polynomials, so that the bounds on the coefficients and their derivatives 
can be also carried over.   Thanks to this we obtain explicit bounds on the Lipschitz constant of the roots.  
As a by-product we give a new proof of the aforementioned result of 
Colombini, Orr\'u, and Pernazza on the existence of $C^1$-roots; see Theorem \ref{mainC1}. 

Note that the statements of Theorem~\ref{main}, Theorem~\ref{main2}, and Theorem~\ref{mainC1} are best possible in the 
following sense. 
If the coefficients are just $C^{p-1,1}$ then the roots need not admit a differentiable 
parameterization. Moreover, the roots can in general not be parameterized by $C^{1,\al}$-functions for any $\al>0$ 
even if the coefficients are $C^\infty$. 
Some better conclusions can be obtained if additional assumptions are made; see \cite{AKLM98}, \cite{BBCP06}, 
\cite{BonyColombiniPernazza06}, \cite{BonyColombiniPernazza10}, \cite{RainerOmin}, \cite{RainerFin}.

\begin{convention*}
  We will denote by $C(n,\ldots)$ \emph{any} constant depending only on $n, \ldots$;
  it may change from line to line. 
  Specific constants will bear a subscript like $C_1(n)$ or $C_2(n)$.  
\end{convention*}

 % For $k=n$, \eqref{A.4} is understood to hold almost everywhere, by Rademacher's theorem. 

%-----------------------------------------------------------------------------------------------------------------------
\section{Bronshtein's theorem}
%-----------------------------------------------------------------------------------------------------------------------

Let $I \subseteq \R$ be an open interval and consider a monic polynomial 
\begin{equation*}
  P_a(t)(Z) = P_{a(t)}(Z) = Z^n + \sum_{j=1}^n a_j(t) Z^{n-j}, \quad t \in I.
\end{equation*}
We say that $P_a(t)$, $t\in I$, 
is a \emph{$C^{p-1,1}$-curve of hyperbolic polynomials} if $(a_j)_{j=1}^n \in C^{p-1,1}(I,\R^n)$ and 
all roots of $P_a(t)$ are real 
for each $t \in I$.  

Note that ordering the roots of a hyperbolic polynomial $P_a(Z)=Z^n + \sum_{j=1}^n a_j Z^{n-j}$
increasingly, $\la^{\uparrow}_1(a) \le \la^{\uparrow}_2(a) \le \cdots \le \la^{\uparrow}_n(a)$, 
provides a continuous mapping $\la^{\uparrow} = (\la^{\uparrow}_j)_{j=1}^n : H_n \to \R^n$ 
on the space of 
hyperbolic polynomials of degree $n$, see e.g.~\cite[Lemma 4.1]{AKLM98}, 
which can be identified with a closed semialgebraic subset 
$H_n \subseteq \R^n$, see e.g.\ \cite{Procesi78}.

By a \emph{system of the roots} of $P_a(t)$, $t \in I$, we mean any $n$-tuple $\la = (\la_j)_{j=1}^n : I \to \R^n$
satisfying
\[
    P_a(t)(Z) = \prod_{j=1}^n (Z - \la_j(t)), \quad t \in I.   
\] 
Note that any \emph{continuous root} $\mu_1$ of $P_a(t)$, $t \in I$, i.e., $\mu_1 \in C^0(I,\R)$ and 
$P_a(t)(\mu_1(t)) = 0$ for all $t \in I$, can be completed to a continuous system of the roots $\mu=(\mu_j)_{j=1}^n$,
cf.\ \cite[Lemma 6.17]{RainerN}.

\begin{theorem}[Bronshtein's theorem] \label{main}
  Let $P_a(t)$, $t \in I$, be a $C^{p-1,1}$-curve of hyperbolic polynomials of degree $n$, where 
  $p$ is the maximal multiplicity of the roots of $P_a$.  
  Then any continuous root of $P_a$ is locally Lipschitz. 
  
  Moreover if $p=n$ then  for any pair of intervals $I_0 \Subset I_1\Subset I$ and for any continuous root $\la (t)$ its 
  Lipschitz constant can be bounded as follows 
  \begin{align}\label{theorem1bounds}
  \on{Lip}_{I_{0}}(\la)  
		 &  \le C(n,I_0,I_1)\, \big(\max_i \|a_i\|^{\frac 1 i}_{C^{n-1,1}(\overline I_1)}\big) 
		  \\
\notag	
		 &  \le \tilde C (n,I_0,I_1)\, \big(1+\max_i \|a_i\|_{C^{n-1,1}(\overline I_1)}\big) ,
\end{align}
where the constants $C(n,I_0,I_1)$, $\tilde C (n,I_0,I_1)$ depend only on $n$ and the intervals $I_0,I_1$.  
(More precise bounds are stated in Subsection \ref{subsectionbounds1}.)

If $p<n$ then  there exist  uniform bounds on the Lipschitz constant provided 
 the multiplicities of roots are at most $p$ ``in a uniform way''.   These bounds are stated in Subsection~\ref{subsectionbounds2}.
 \end{theorem}

For an open subset $U \subseteq \R^m$ and $p \in \N_{\ge 1}$, we denote by $C^{p-1,1}(U)$ 
the space of all functions $f \in C^{p-1}(U)$ so that each partial derivative $\p^\al f$ of order
$|\al|=p-1$ is locally Lipschitz. It is a Fr\'echet space with the following system of seminorms,
\[
  \|f\|_{C^{p-1,1}(K)} = \|f\|_{C^{p-1}(K)} + \sup_{|\al|=p-1}\on{Lip}_K(\p^\al f), 
  \quad \on{Lip}_K(f) = \sup_{\substack{x,y \in K\\ x \ne y}} \frac{|f(x)-f(y)|}{\|x-y\|},
\]
where $K$ ranges over (a countable exhaustion of) the compact subsets of $U$; 
on $\R^m$ we consider the $2$-norm $\|~\|=\|~\|_2$. 
%A subset $A$ of $C^{p-1,1}(U)$ is bounded if and only if $\sup_{f \in A} \|f\|_{C^{p-1,1}(K)} < \infty$ for all $K$. 

By Rademacher's theorem, the partial derivatives of order $p$ of a function $f \in C^{p-1,1}(U)$ exist almost everywhere 
and coincide almost everywhere with the corresponding weak partial derivatives.

Theorem~\ref{main} readily implies the following multiparameter version.

\begin{theorem} \label{main2}
  Let $U\subseteq \R^m$ be open and 
  let $P_a(x)$, $x \in U$, be a $C^{p-1,1}$-family of hyperbolic polynomials of degree $n$, 
  where $p$ is the maximal multiplicity of the roots of $P_a$. 
  Then any continuous  root of $P_a$ is locally Lipschitz.

  Moreover, if $p=n$ for any pair of relatively compact subsets $U_0 \Subset U_1\Subset U$ and for any continuous root $\la (x)$ its 
  Lipschitz constant can be bounded as follows 
  \begin{align}\label{theorem2bounds}
  \on{Lip}_{U_0}(\la)  
     &  \le C(m,n,U_0,U_1)\, \big(\max_i \|a_i\|^{\frac 1 i}_{C^{n-1,1}(\overline U_1)}\big) 
      \\
  \notag  
     &  \le \tilde C (m,n,U_0,U_1)\, \big(1+\max_i \|a_i\|_{C^{n-1,1}(\overline U_1)}\big) ,
  \end{align}
  where the constants $C(m,n,U_0,U_1)$, $\tilde C (m,n,U_0,U_1)$ depend only on $m$, $n$, and the sets $U_0,U_1$.
\end{theorem}

\begin{proof}
  Let $\la$ be a continuous root of $P_a$.  
  Without loss of generality we may assume that $U_0$ and $U_1$ are open boxes 
  parallel to the coordinate axes, $U_i = \prod_{j=1}^m I_{i,j}$, $i=0,1$, with $I_{0,j} \Subset I_{1,j}$ for all $j$.
  Let $x,y \in U_0$ and set $h:= y-x$.
  Let $\{e_j\}_{j=1}^m$ denote the standard unit vectors in $\R^m$.
  For any $z$ in the orthogonal projection of $U_0$ on the hyperplane $x_j =0$ 
  consider the function $\la_{z,j} : I_{0,j} \to \R$ defined by 
  $\la_{z,j}(t) := \la(z + t e_j)$. 
  By Theorem~\ref{main}, each $\la_{z,j}$ is Lipschitz and $C := \sup_{z,j} \on{Lip}_{I_{0,j}}(\la_{z,j}) < \infty$.
  Thus 
  \[
    |\la(x)-\la(y)| 
    \le \sum_{j=0}^{m-1} \Big|\la\big(x+\sum_{k=1}^j h_k e_k\big) - \la\big(x+\sum_{k=1}^{j+1} h_k e_k\big)\Big| 
    \le C \|h\|_1 \le C \sqrt{m} \|h\|_2.
  \] 
  The bounds \eqref{theorem2bounds} follow from \eqref{theorem1bounds}.
\end{proof}

\begin{corollary}
  Let $U\subseteq \R^m$ be open.
	The push forward $(\la^{\uparrow})_* : C^{n-1,1}(U,\R^n) \supseteq C^{n-1,1}(U,H_n) \to C^{0,1}(U,\R^n)$ 
  is bounded.
\end{corollary}

%\todo[inline]{Is this corollary necessary?  }

Next we suppose that $P_a(t)$, $t \in I$, is a $C^{p}$-curve of hyperbolic polynomials of degree $n$, where 
$p$ is the maximal multiplicity of the roots of $P_a$.  
Then the roots can be chosen $C^1$. 
We will give a new proof of this recent result of \cite{ColombiniOrruPernazza12},  see Theorem~\ref{mainC1}.
    
For a function $f(t)$ we denote by $f'^-(t_0)$ (resp. $f'^+(t_0)$) the left (resp.\ right) derivative of $f$ 
at the point $t_0$.

\begin{theorem} \label{mainC1}
  Let $P_a(t)$, $t \in I$, be a $C^{p}$-curve of hyperbolic polynomials of degree $n$, where 
  $p$ is the maximal multiplicity of the roots of $P_a$.  
  Then: 
  \begin{enumerate}
  \item
  Any continuous root $\la(t)$ of $P_a$ has both one-sided 
  derivatives at every $t\in I$.  
  \item
  These derivatives are continuous: for every $t_0\in I$  we have 
  $$
  \lim _{t\to t_0^-} \lambda'^\pm (t) =   \lambda'^- (t_0)   \quad  \lim _{t\to t_0^+} \lambda'^\pm (t) =   \lambda'^+ (t_0) .
 $$
 \item
There exists a differentiable system of the roots.  
\item 
Any differentiable  root is $C^1$.  
 \end{enumerate}
\end{theorem}

%-----------------------------------------------------------------------------------------------------------------------
\section{Preliminaries}
%-----------------------------------------------------------------------------------------------------------------------

%-----------------------------------------------------------------------------------------------------------------------
\subsection{Tschirnhausen transformation}
%-----------------------------------------------------------------------------------------------------------------------

A monic polynomial 
\begin{equation*}
P_{a}(Z) = Z^n + \sum_{j=1}^n a_j Z^{n-j}, \quad  a=(a_1, ... ,a_n) \in \R_a^n 
\end{equation*}
 is said to be in \emph{Tschirnhausen form} if $a_1=0$.  Every $P_a$ can be transformed to such a form 
 by the substitution $Z \mapsto Z-\frac{a_1}n$, which we refer to as 
\emph{Tschirnhausen transformation}, 
\begin{equation}\label{tschirnhausenform}
P_{\tilde a}(Z) = P_a  ( Z-\frac{a_1}n  ) = Z^n   + \sum_{j=2}^n \ti a_j Z^{n-j}, \quad   \ti a=(\ti a_2, ... ,\ti a_n) \in \R_{\ti a}^{n-1}. 
\end{equation}
We identify the set of monic real polynomials $P_a$ of degree $n$ with $\R^{n}_a$, where $a=(a_1,a_2,\ldots,a_n)$,   
and those in Tschirnhausen form with $\R^{n-1}_{\tilde a}$.  
In what follows we write the effect of the Tschirnhausen transformation on a polynomial $P_a$ simply by adding tilde, $P_{\ti a}$.  
 
Thus let $P_{\ti a}$ be a monic polynomial in Tschirnhausen form. 
Then 
\[
	s_2 = -2\ti a_2 = \la_1^2 + \cdots + \la_n^2,	
\] 
where the $s_i$ denote the Newton polynomials in the roots $\la_j$ of $P_a$.
Thus, for a hyperbolic polynomial $P_{\ti a}$ in Tschirnhausen form,
\[
	s_2 = -2\ti a_2 \ge 0.	
\]

\begin{lemma} \label{lem:hyp}
  The coefficients of a hyperbolic polynomial $P_{\ti a}$ in Tschirnhausen form satisfy
  \[
    |\ti  a_i|^{\frac{1}{i}} \le  |s_2|^{\frac{1}{2}} = \sqrt{2}\, |\ti  a_2|^{\frac 1 2}, \quad i=2,\ldots,n.
  \]
\end{lemma}

\begin{proof}
  Newton's identities give $|\ti  a_i| \le \frac{1}{i} \sum_{j=2}^i |s_j||\ti a_{i-j}|$, where $\ti a_0=1$, which 
  together with 
  \begin{equation} \label{eq:hyp}
    |s_i|^{\frac{1}{i}} \le  |s_2|^{\frac{1}{2}}, \quad i=2,\ldots,n,
  \end{equation}
  will imply the result by induction on $i$. 
  To show \eqref{eq:hyp} we note that it is equivalent to 
  \begin{equation} \label{eq:roots}
    (\la_1^i + \cdots + \la_n^i)^2 \le  (\la_1^2 + \cdots + \la_n^2)^i.
  \end{equation}
  Each mixed term $\la_\ell^i \la_m^i$ on the left-hand side of \eqref{eq:roots} may be estimated 
  by the sum of all $\la_\ell^a \la_m^b$ 
  terms with $a,b>0$ on the right-hand side of \eqref{eq:roots}, in fact
  \[
    2 \la_\ell^i \la_m^i = 2 \la_\ell^2 \la_m^2 \la_\ell^{i-2} \la_m^{i-2}
    \le \la_\ell^2 \la_m^2 (\la_\ell^{2(i-2)} + \la_m^{2(i-2)})
    \le \sum_{j = 1}^{i-1} \binom{i}{j} \la_\ell^{2j} \la_m^{2(i-j)}.  
  \]
  This implies the statement.
\end{proof}

%-----------------------------------------------------------------------------------------------------------------------
\subsection{Splitting}   \label{ssec:split}
%-----------------------------------------------------------------------------------------------------------------------

The following well-known lemma (see e.g.\ \cite{AKLM98} or \cite{BM90}) is  an easy  consequence of the inverse function theorem.  

\begin{lemma} \label{split}
Let $P_a = P_b P_c$, where $P_b$ and $P_c$ are monic complex polynomials without common root.
Then for $P$ near $P_a$ we have $P = P_{b(P)} P_{c(P)}$
for analytic mappings $\R_a^n  \ni P  \mapsto b(P)  \in \R_b^{\deg P_b}$ and $\R_a^n  \ni P  \mapsto c(P) \in \R_c^{\deg P_c} $,
defined for $P$
near $P_a$, with the given initial values.
\end{lemma}

\begin{proof}
The product $P_a = P_b P_c$ defines on the coefficients a polynomial mapping $\vh$ such that $a = \vh(b,c)$, 
where $a=(a_i)$, $b=(b_i)$, and $c=(c_i)$. The Jacobian determinant  
$\det d\vh(b,c)$ equals the resultant of $P_b$ and $P_c$ which is nonzero by assumption.   Thus $\varphi$ can be inverted locally. 
\end{proof}

If $P_{\ti a}$ is in Tschirnhausen form and $\tilde a\ne 0$ then, the sum of its roots being equal to zero, it always splits.   
The space of hyperbolic polynomials of degree $n$ in Tschirnhausen form can be identified with a 
closed semialgebraic subset $H_n$ of $\R^{n-1}_{\ti a}$.  By Lemma~\ref{lem:hyp}, 
the set $H_n^0:= H_n \cap \{\ti a_2 = -1\}$   
is compact.  

Let $p \in H_n \cap \{\ti a_2 \ne 0\}$.   
Then the polynomial
\[
  Q_{\underline a}(Z) := |\ti a_2|^{- \frac n 2} P_{\ti  a} (|\ti a_2|^{\frac 1 2} Z) 
  = Z^n -Z^{n-2} + |\ti  a_2|^{- \frac 3 2} \ti a_3 Z^{n-3} + \cdots + |\ti a_2|^{- \frac n 2} \ti  a_n 
\]
is hyperbolic and, by Lemma~\ref{split}, 
it \emph{splits}, i.e., $Q_{\underline a} = Q_{\underline b} Q_{\underline c}$ and $\deg Q_{\underline b}, \deg Q_{\underline c} < n$, 
on some open ball $B_p(r)$ centered at $p$.
Thus, there exist real analytic functions $\ps_i$ so that, on $B_p(r)$,
\[
   \underline b_i = \ps_i \big(|\ti a_2|^{-\frac{3}{2}} \ti a_3, \ldots, |\ti a_2|^{-\frac{n}{2}} \ti a_n\big), 
    \quad i = 1,\ldots,\deg P_b; 
\] 
likewise for $\underline c_j$. The splitting $Q_{\underline a} = Q_{\underline b} Q_{\underline c}$ induces a splitting 
$P_{\ti a} = P_b P_c$, where 
\begin{equation} \label{eq:3}
    b_i = |\ti a_2|^{\frac{i}{2}} \ps_i \big(|\ti a_2|^{-\frac{3}{2}} \ti a_3, \ldots, |\ti a_2|^{-\frac{n}{2}} \ti a_n\big), 
    \quad i = 1,\ldots,\deg P_b;
  \end{equation}
likewise for $c_j$. 
Shrinking $r$ slightly, we may assume that all partial derivatives of $\ps_i$ are separately bounded on $B_p(r)$. 
We denote by $\tilde b_j$ the coefficients of the polynomial $P_{\tilde b}$ resulting from $P_b$ by the  
Tschirnhausen transformation.

\begin{lemma} \label{lem:b2}
  In this situation we have
    $|\tilde b_2| \le 2n |\ti a_2|$.
\end{lemma}
 
\begin{proof}
  Let $(\la_j)_{j=1}^k$ denote the roots of $P_b$ and $(\la_j)_{j=1}^n$ those of $P_a$.			
  Then, as $|b_1| \le \sum_{j=1}^k |\la_j| \le (k \sum_{j=1}^k \la_j^2)^{1/2}$ and thus 
  $|\la_j||b_1| \le k \sum_{j=1}^k \la_j^2$,  
  \begin{align*}
    2 |\tilde b_2| &= \sum_{j=1}^k \Big(\la_j + \frac{b_1}{k}\Big)^2 
    \le \frac{1}{k^2} \sum_{j=1}^k (k^2 \la_j^2 + b_1^2 + 2 k |\la_j||b_1|) \\
    & \le \frac{1}{k^2} \sum_{j=1}^k \big(k^2 \la_j^2 + k \sum_{\ell=1}^k \la_\ell^2 + 2 k^2 \sum_{\ell=1}^k \la_\ell^2\big) 
    %\\& 
    =  2(k+1) \sum_{j=1}^k \la_j^2  \le 2 n \sum_{j=1}^n \la_j^2 = 4n |\ti a_2|, 
  \end{align*}
  as required.
\end{proof}

%-----------------------------------------------------------------------------------------------------------------------
\subsection{Coefficient estimates} 
%-----------------------------------------------------------------------------------------------------------------------

We shall need the following estimates. (Here it is convenient to number the coefficients in reversed order.)

\begin{lemma} \label{lem:interpol}
  Let $P(x) = a_0 + a_1 x + \cdots + a_n x^n \in \C[x]$ satisfy 
  $|P(x)| \le A$ for $x\in [0,B] \subseteq  \R$. Then 
  \[
    |a_j| \le (2n)^{n+1} A B^{-j}, \quad j=0,\ldots,n.
  \]
\end{lemma}

\begin{proof}
  We show the lemma for $A=B=1$.
  The general statement follows by applying this special case to the polynomial $A^{-1} P(By)$, $y=B^{-1}x$.
  Let $0 = x_0 < x_1 < \cdots < x_n=1$ be equidistant points. By Lagrange's interpolation formula 
  (e.g.\ \cite[(1.2.5)]{RS02}), 
  \[
    P(x) = \sum_{k=0}^n P(x_k) \prod_{\substack{j=0\\ j\ne k}}^n \frac{x-x_j}{x_k-x_j},
  \]
  and therefore
  \[
    a_j = \sum_{k=0}^n P(x_k) \prod_{\substack{j=0\\ j\ne k}}^n (x_k-x_j)^{-1} (-1)^{n-j} \si^k_{n-j},
  \]
  where $\si^k_j$ is the $j$th elementary symmetric polynomial in $(x_\ell)_{\ell \ne k}$.
  The statement follows. 
\end{proof}

A better constant can be obtained using Chebyshev polynomials; cf.\ \cite[Thm.~16.3.1-2]{RS02}. 

%-----------------------------------------------------------------------------------------------------------------------
\subsection{Consequences of Taylor's theorem} 
%-----------------------------------------------------------------------------------------------------------------------

The following two lemmas are classical.  We include them for the reader's convenience. 

\begin{lemma} \label{Glaeser}
  Let $I\subseteq \R$ be an open interval and 	
  let $f \in C^{1,1}(\overline I)$ be nonnegative or nonpositive. 
  For any $t_0 \in I$  and $M>0$ such that $I_{t_0}(M^{-1}) := \{t : |t-t_0| < M^{-1}|f(t_0)|^{\frac 1 2}\} \subseteq I$ 
  and  $M\ge (\on{Lip}_{I_{t_0} (M^{-1})}(f')) 
  ^{\frac 1 2} $ 
  we have 
  \begin{align*}
    |f'(t_0)|& \le\big(M + M^{-1} \on{Lip}_{I_{t_0}(M^{-1})}(f')\big)  |f(t_0)|^{\frac{1}{2}}    
    \le 2  M  |f(t_0)|^{\frac{1}{2}}. 
  \end{align*}
\end{lemma}

\begin{proof}
	Suppose that $f$ is nonnegative; otherwise consider $-f$.
  It follows that the inequality holds true at the zeros of $f$. Let us assume that $f(t_0)>0$. 
  The statement follows from  
  \[
   0 \le  f(t_0+h) = f(t_0) + f'(t_0) h + \int_0^1 (1-s) f''(t_0+ h s)\, ds\, h^2
  \]
  with $h = \pm M^{-1}|f(t_0)|^{\frac{1}{2}}$. 
\end{proof}

\begin{lemma} \label{taylor} 
  Let $f\in C^{m-1,1}(\overline I)$.  There is a universal constant $C(m)$ such that for all $t\in I$ 
  and  $ k = 1,\ldots,m$,  
    \begin{align}\label{eq:1}  
    |f^{(k)}(t) | \le C(m) |I|^{-k} \bigl(\| f \|_{L^\infty(I)}+  \on{Lip}_{I}(f^{(m-1)})  |I|^m
    \bigr)  .  
  \end{align}
\end{lemma}

\begin{proof}
We may suppose that $I=(-\delta, \delta)$. If $t \in I$ then at least one of the two intervals
 $[t,t\pm \delta )$, say $[t,t+ \delta )$, is included in $I$.  By Taylor's formula, for $t_1\in [t,t +\delta )$, 
    \begin{align*}
      \Big|\sum_{k=0}^{m-1}  \frac{{f}^{(k)}(t)}{k!} (t_1-t)^k\Big| 
       & \le |f(t_1)| + | \int_0^1 \frac {(1-s)^{m-1}}{(m-1)!} f^{(m)}(t+s(t_1-t))\, ds\, (t_1-t)^m|  \\
      & \le \| f \|_{L^\infty(I)}+  \on{Lip}_{I}(f^{(m-1)})  \delta^m ,        
    \end{align*}
   and for $k \le m-1$  we may conclude \eqref{eq:1} by Lemma~\ref{lem:interpol}.   
    For $k=m$,  \eqref{eq:1} is trivially satisfied.  
\end{proof}

%-----------------------------------------------------------------------------------------------------------------------
\section{Proof of Theorem~\ref{main}}
%-----------------------------------------------------------------------------------------------------------------------

%-----------------------------------------------------------------------------------------------------------------------
\subsection{First reductions} \label{ssec:red}
%-----------------------------------------------------------------------------------------------------------------------

We assume that the maximal multiplicity $p$ of the roots equals the degree $n$ of $P_a$. 
If $p<n$ then we may use Lemma~\ref{split} to split $P_a$ locally in factors that have this property.  We discuss it in more detail at the end of the proof.

So let $P_{a} (t)$, $t \in I$, be a $C^{n-1,1}$-curve of hyperbolic polynomials of degree $n$. 
Without loss of generality we may assume that $n\ge 2$ and that $P_a = P_{\ti a}$ is in Tschirnhausen form.
Let $(\la_j(t))_{j=1}^n$, 
$t \in I$, be any continuous system of the roots of $P_{\ti a}$. Then
\[
\ti   a_2(t)=0  \quad\Longleftrightarrow\quad  \la_1(t)=\cdots= \la_n(t)=0.
\] 
We shall show that, for any relatively compact open subinterval $I_0 \Subset I$ and any 
$t_0 \in I_0 \setminus {\ti  a_2}^{-1}(0)$, 
there exists a neighborhood $I_{t_0}$ of $t_0$ in $I_0 \setminus {\ti  a_2}^{-1}(0)$ so that each $\la_j$ is Lipschitz on $I_{t_0}$ 
and the Lipschitz constant $\on{Lip}_{I_{t_0}}(\la_j)$ satisfies
\[
  \on{Lip}_{I_{t_0}}(\la_j) \le C(n,I_0,I_1)\, \big(\max_i \|\ti a_i\|_{C^{n-1,1}(\overline I_1)}^{\frac 1 i}\big), 
\]
where $I_1$ is any open interval 
satisfying $I_0 \Subset I_1 \Subset I$.  Here, recall, $ C(n,I_0,I_1)$ stands for a universal constant depending only on $n$, $I_0$, and $I_1$.

This will imply Theorem~\ref{main} by the following lemma. 

\begin{lemma} \label{lem:ext1}
  Let $I \subseteq \R$ be an open interval.
  If $f \in C^0(I)$ and each $t_0 \in I  \setminus f^{-1}(0)$ has a neighborhood $I_{t_0} \subseteq  I \setminus f^{-1}(0)$ 
  so that $L := \sup_{t_0 \in I \setminus f^{-1}(0)} \on{Lip}_{I_{t_0}}(f) < \infty$, 
  then $f$ is Lipschitz on $I$ and $\on{Lip}_{I}(f) = L$.
\end{lemma}

\begin{proof}
  Let $t,s \in I$. It is easy to see that $|f(t)-f(s)| \le L |t-s|$ if $t$ and $s$ belong to the same connected 
  component $J$ of $I  \setminus f^{-1}(0)$.
  By continuity, this estimate also holds on the closed interval $\overline J$. If $t \in \overline J_1$ and 
  $s \in \overline J_2$, $t<s$, and $\overline J_1 \cap \overline J_2 = \emptyset$, let $r_i$ be the endpoint 
  of $\overline J_i$ so that $s \le r_1 < r_2\le t$. 
  Then 
  \[
    |f(t)-f(s))| \le |f(t)-f(r_2)| + |f(r_1)-f(s)| \le L |t-s|. 
  \] 
  Clearly, $\on{Lip}_{I}(f) = L$.
\end{proof}

%-----------------------------------------------------------------------------------------------------------------------
\subsection{Convenient assumption}
%-----------------------------------------------------------------------------------------------------------------------

The proof of the statement in Subsection \ref{ssec:red} will be carried out by induction on the degree of $P_a$. 
We replace the assumption of Theorem~\ref{main} by a new assumption 
that will be more convenient for the inductive step. Before we state it we need a bit of notation. 

For open intervals $I_0$ and $I_1$ so that $I_0 \Subset I_1 \Subset I$, we set 
\[
	I_i':= I_i \setminus {\ti a_2}^{-1}(0), \quad i=0,1.
\]
For $t_0 \in I_0'$ and $r>0$ consider the interval 
\[
	I_{t_0}(r) := \big(t_0 -r |\ti a_2(t_0)|^\frac{1}{2},t_0 + r |\ti a_2(t_0)|^\frac{1}{2}\big).
\]

\begin{assumption*} 
	Let $I_0 \Subset I_1$ be open intervals.	  Suppose that $(\ti a_i)_{i=2}^n \in C^{n-1,1}(\overline I_1,\R^{n-1})$ are 
	the coefficients of a hyperbolic polynomial $P_{\ti a}$ of degree $n$ 
  in Tschirnhausen form. Assume that there is a constant $A>0$,  
  so that for all $t_0 \in I_0'$, $t \in I_{t_0}(A^{-1})$, $i = 2,\ldots,n$, $k = 0,\ldots,n$,
  \begin{gather}
    \tag{A.1}\phantomsection\label{A.2} I_{t_0}(A^{-1}) \subseteq I_1,\\ 
    \tag{A.2}\phantomsection\label{A.3} 2^{-1} \le \frac{\ti  a_2(t)}{\ti a_2(t_0)} \le 2, \\ 
    \tag{A.3}\phantomsection\label{A.4} |{\ti a_i}^{(k)}(t)| \le C(n)\,  A^{k}\, |\ti a_2(t)|^{\frac{i-k}{2}},
  \end{gather}
  where $C(n)$ is a universal constant.  
  For $k=n$, \eqref{A.4} is understood to hold almost everywhere, by Rademacher's theorem. 
  %From now on we use this convention without mentioning.
\end{assumption*}

  Condition \eqref{A.4} implies that  
  \begin{equation} \tag{A.4}\label{A.5}
    \big|\p_t^k \big(|\ti a_2|^{- \frac i 2} \ti  a_i\big)(t)\big| \le C(n)\,  A^k\, |\ti a_2(t)|^{-\frac k 2}.
  \end{equation}
  More generally, if we assign $\ti a_i$ the weight $i$ and $|\ti a_2|^{\frac 1 2}$ the weight $1$ and let 
  $L(x_2,\ldots,x_n,y) \in \R[x_2,\ldots,x_n,y,y^{-1}]$ be weighted homogeneous of degree $d$, 
  then    
  \[
    \big|\p_t^k L\big(\ti a_2,\ldots, \ti a_n,|\ti a_2|^{\frac 1 2}\big)(t)\big| \le C(n,L)\, A^k\, |\ti a_2(t)|^{\frac{d-k}{2}}.
  \] 

%-----------------------------------------------------------------------------------------------------------------------
\subsection{Inductive step}  \label{ssec:indstep}
%-----------------------------------------------------------------------------------------------------------------------

Let $P_{\ti a}$, $I_0$, $I_1$, $A$,  $t_0$ be as in Assumption.  
We will show by induction on $\deg P_{\ti a}$ that any continuous system of the roots of $P_{\ti a}$ %that satisfies this assumption 
is Lipschitz on $I_0$ with Lipschitz constant bounded from above by $C(n)\, A$. 
First we establish the following.
\begin{itemize}
 	\item For some constant $C_1(n)>1$, the polynomial $P_{\ti a}(t)$ splits on the interval $I_{t_0}(C_1(n)^{-1}A^{-1})$,
 	 that is we have $P_{\ti a} (t) = P_b(t) P_c(t)$, where $P_b$ and $P_c$ are 
 	$C^{n-1,1}$-curves of hyperbolic polynomials of degree strictly smaller than $n$.
 	\item After applying the Tschirnhausen transformation $P_b \leadsto P_{\tilde b}$, the coefficients 
 	$(\tilde b_i)_{i=2}^{\deg P_b}$ satisfy \eqref{A.2}--\eqref{A.4} for suitable neighborhoods $J_0$, $J_1$ of $t_0$, 
 	and a constant $B=C(n)\, A$ in place of $A$. 
\end{itemize} 
%This will allow us to conclude by induction on $n$. 

\medskip
We restrict our curve of hyperbolic polynomials $P_{\ti a}$ to $I_{t_0}(A^{-1})$ and consider 
\[
	\underline a:= \big(-1,|\ti  a_2|^{-\frac 3 2} \ti a_3,\ldots,|\ti  a_2|^{-\frac n 2} \ti a_n\big) : 
	I_{t_0}(A^{-1}) \to  \R^{n-1}_{\underline a }.
\]
Then $\underline a$ is continuous, by \eqref{A.3}, and bounded, by Lemma~\ref{lem:hyp}. 
Moreover, by \eqref{A.5} and \eqref{A.3}, there is a universal constant $C_1(n)$ so that, for $t \in I_{t_0}(A^{-1})$,
\begin{equation}\label{eq:barader}
	\|\underline  a'(t)\| \le C_1(n)\, A\, |\ti a_2(t_0)|^{-\frac 1 2}.  	
\end{equation}  
According to Subsection \ref{ssec:split}, choose a finite cover of  $H_n^0 $   
 by  open balls $B_{p_\al}(r)$, $\al \in \De$, on which we 
have a splitting $P_{\ti a} = P_b P_c$ with coefficients of $P_b$ given by \eqref{eq:3}.
There exists $r_1>0$ such that for any $p \in H_n^0 $ there is $\al \in \De$ so that 
$B_p(r_1) \subseteq B_{p_\al}(r)$; $2r_1$ is a Lebesgue number of the cover $\{B_{p_\al}(r)\}_{\al \in \De}$.     
Then, if $C_1(n)$ is the constant from \eqref{eq:barader}, 
\begin{equation} \label{eq:J1}
  J_1 := I_{t_0}(r_1 C_1(n)^{-1} A^{-1})  \subseteq {\underline a}^{-1} (B_{\underline a(t_0)}(r_1)),
\end{equation}
and on $J_1$ we have a splitting $P_{\ti a}(t) = P_b(t) P_c(t)$ with $b_i$ given by \eqref{eq:3}. Fix $r_0 < r_1$ and let 
\begin{equation} \label{eq:J0}
  J_0 := I_{t_0}(r_0 C_1(n)^{-1} A^{-1}).
\end{equation}
(Here we assume without loss of generality that $r_1\le C_1(n)$.)

Let us show that the coefficients $(\tilde b_i)_{i=2}^{\deg P_b}$ of $P_{\tilde b}$ satisfy \eqref{A.2}--\eqref{A.4} 
for the intervals $J_1$ and $J_0$ from \eqref{eq:J1} and \eqref{eq:J0}. 
To this end we set 
\[
  J_i':= J_i \setminus \tilde b_2^{-1}(0), \quad i=0,1,
\]
consider, 
for $t_1 \in J_0'$ and $r>0$, 
\[
  J_{t_1}(r) := \big(t_1 -r |\tilde b_2(t_1)|^\frac{1}{2},t_1 + r |\tilde b_2(t_1)|^\frac{1}{2}\big),
\]
and prove the following lemma.

\begin{lemma}\label{lem:induction}
  There exists a constant $\ti C = \ti C(n,r_0,r_1) >1$ such that for $B= \tilde C A $ and  
  for all $t_1 \in J_0'$, $t \in J_{t_1}(B^{-1})$, $i = 2,\ldots,\deg P_b$, $k = 0,\ldots,n$,
  \begin{gather}
  %  \tag{B.1}\phantomsection\label{B.1} |\tilde b_2(t_1)| \le B \\
    \tag{B.1}\phantomsection\label{B.2} J_{t_1}(B^{-1}) \subseteq J_1, \\ 
    \tag{B.2}\phantomsection\label{B.3} 2^{-1} \le \frac{\tilde b_2(t)}{\tilde b_2(t_1)} \le 2, \\ 
    \tag{B.3}\phantomsection\label{B.4} |{\tilde b_i}^{(k)}(t)| \le  C(n)\, B^k\, |\tilde b_2(t)|^{\frac{i-k}{2}}. 
  \end{gather}
  for some universal constant $C(n)$.
\end{lemma}

\begin{proof}
  If 
  \begin{equation*} \label{eq:D1}
    B \ge  (r_1-r_0)^{-1}\, 2 \sqrt{n}\, C_1(n)\,  A,
  \end{equation*}
  then by Lemma~\ref{lem:b2} and \eqref{A.3},
  \[
    B^{-1}|\tilde b_2(t_1)|^{\frac 1 2}  
    \le (r_1-r_0)\, C_1(n)^{-1}\,  A^{-1}\, |\ti  a_2(t_0)|^{\frac 1 2},
  \]
  and hence \eqref{B.2} follows from \eqref{eq:J1} and \eqref{eq:J0}, since $t_1 \in J_0$.

  Next we claim that, on $J_1$,
  \begin{equation} \label{eq:psi}
    \big|\p_t^k \ps_i \big(|\ti a_2|^{-\frac 3 2} \ti a_3,\ldots,|\ti a_2|^{-\frac n 2} \ti a_n\big)\big| 
    \le C(n)\,  A^k\, |\ti a_2|^{-\frac k 2}.
  \end{equation}
  To see this we differentiate the following equation $(k-1)$ times, apply  
  induction on $k$, and use \eqref{A.5},
  \begin{align}\label{derivatives}
    \p_t \ps_i \big(|\ti a_2|^{-\frac 3 2} \ti a_3,\ldots,|\ti a_2|^{-\frac n 2} \ti a_n\big) 
    = \sum_{j=3}^n (\p_{j-2} \ps_i)(\underline a) \p_t \big(|\ti a_2|^{-\frac j 2} \ti a_j\big);
\end{align}
  recall that all partial derivatives of the $\ps_i$'s are separately bounded on $\underline a(J_1)$ and these bounds are 
  universal. From \eqref{eq:3} and \eqref{eq:psi} we obtain, on $J_1$ and for all $i=1,\ldots,\deg P_b$, $k=0,\ldots,n$,
  \begin{equation} \label{eq:b1}
    |b_i^{(k)}| \le C(n)\,  A^k\, |\ti a_2|^{\frac{i-k}{2}},
  \end{equation}
  thus, as the Tschirnhausen transformation preserves the weights of the coefficients, 
  cf.\ \eqref{A.5},
  \[
    |{\tilde b_i}^{(k)}| \le C(n)\,  A^k\, |\ti a_2|^{\frac{i-k}{2}},
  \]
  and so, by Lemma~\ref{lem:b2},   
  \begin{equation*} \label{eq:i-k<0}
    |{\tilde b_i}^{(k)}| \le C (n)\,  A^k\, |\tilde b_2|^{\frac{i-k}{2}} \quad \text{ if } i-k \le 0.
  \end{equation*} 
% This gives for $k=i$ 
 %  \[
 %   |{\on{Lip}_{J_1} ( \tilde b_i}^{(i-1)}| \le C(n)\,  A^i.  
 % \]  

This shows \eqref{B.4} for $i\le k$.  \eqref{B.4} for $k=0$ follows from Lemma \ref{lem:hyp}.   
\eqref{B.3} and the remaining inequalities of \eqref{B.4}, i.e., for  $0<k<i$, follow now from Lemma \ref{shortcut} below.  
  \end{proof}

\begin{lemma}\label{shortcut}
There exists a constant $C(n)\ge 1$ such that the following holds.  
If \eqref{A.2} and   \eqref{A.4}  for  $k=0 $ and $k=i$, $i=2, ... ,n$,
  are satisfied, then so are \eqref{A.3} and  \eqref{A.4}  for  $k<i$, $i=2, ... ,n$,  after 
  replacing  $A$ by $C(n) A$.   
\end{lemma}

\begin{proof}
By assumption, $\on{Lip}_{I_{t_0}(A^{-1})}(\ti a_2'))\le C(n) A^2$.  Thus, by Lemma \ref{Glaeser} 
for $f=\ti a_2$ and  $M = C(n)^{\frac 1 2} A$, we 
get 
$$
|\ti a'_2(t_0)| \le 2 M |\ti a_2 (t_0)|^{\frac 1 2}.  
$$
It follows that, for $t\in I_{t_0}((6M)^{-1})$, 
\begin{align}\label{long}
	\frac{|\ti a_2(t) - \ti a_2(t_0)|}{|\ti a_2(t_0)|} \le 
    	\frac{|\ti a_2'(t_0)|}{|\ti a_2(t_0) |}|t-t_0| + \int_0^1 (1-s) |\ti a_2''(t_0 + s (t-t_0))| ds\, \frac{|t-t_0|^2}{|\ti a_2(t_0)|} 
     	\le \frac{1}{2}
  	\end{align}
That	implies \eqref{A.3}.  
  The other inequalities follow from Lemma \ref{taylor}.  
\end{proof}

%-----------------------------------------------------------------------------------------------------------------------
\subsection{End of inductive step}
%-----------------------------------------------------------------------------------------------------------------------

In $J_1$, %$I_{t_0}(A^{-1})$, 
any continuous root $\la_j$ of $P_{\ti a}$, where $P_{\ti a}$ is in Tschirnhausen form, is a root of either $P_b$ or 
$P_c$.
Say it is a root of $P_b$. Then it has the form
\begin{align}\label{sumforaroot}
\la_j(t) = - \frac{b_1(t)}{\deg P_b} + \mu_j(t),
\end{align}
where $\mu_j$ is a continuous root of $P_{\tilde b}$ defined on a neighborhood of $t_0$. 
By the inductive assumption we may assume that $\mu_j $ is Lipschitz with Lipschitz constant 
bounded from above by $C(n) B$.
Hence $\la_j$ is Lipschitz with Lipschitz constant 
bounded from above by $C(n) A$ (the constant $C(n)$ changes), as $B= \tilde C\, A$ and by \eqref{eq:b1} for $i=k=1$.  
This ends the inductive step.

%-----------------------------------------------------------------------------------------------------------------------
\subsection{\texorpdfstring{$P_{\ti a}$}{Pa} satisfies Assumption}
%-----------------------------------------------------------------------------------------------------------------------

Now we show that $P_{\ti a}$ always satisfies Assumption.  The choice of $A$ will provide the upper 
bound on the Lipschitz constant of the roots. 

\begin{proposition} \label{lem:1}
	Let $P_{\ti a}(t)$, $t \in I$, be a $C^{n-1,1}$-curve of hyperbolic polynomials of degree $n$ 
  	in Tschirnhausen form, and let $I_0$ and $I_1$ be open intervals satisfying $I_0 \Subset I_1 \Subset I$. 
  	Then its coefficients $(\ti a_i)_{i=2}^n$ satisfy \eqref{A.2}--\eqref{A.4}.
\end{proposition}

\begin{proof} 
	Let $\de$ denote the distance between the endpoints of $I_0$ and those of $I_1$. 
	Set 
	\begin{align}\label{A1A2}
		  A_1:=  \max\big\{  \de^{-1}\|\ti a_2\|_{L^\infty( I_1)}^{\frac{1}{2}}, (\on{Lip}_{ I_1}(\ti a_2'))^{\frac 1 2}  \big\},  \quad 
		  A_2 :=  \max_i \big\{M_i \|\ti a_2\|^{\frac {n-i} 2}_{L^\infty( I_1)} \big\} ^{\frac 1 n} ,		
	\end{align}
	where $M_i =   \on{Lip}_{ I_1}(\ti a_i ^{(n-1)})$.  Then we may choose 
	\begin{align}\label{formulaforA}
		A \ge  A_0= 6 \max \{{A_1}, A_2\}.	
	\end{align}	
	For  \eqref{A.2}--\eqref{A.3} to be satisfied we need only $A\ge 6A_1$.  Indeed, clearly, for $t_0 \in I_0'$, 
	\begin{equation} \label{eq:incl}
		I_{t_0}(A_1^{-1}) \subseteq I_1.
	\end{equation}
	Then Lemma~\ref{Glaeser} implies that 
  	\begin{align*}% \label{eq:1}
    	 |\ti a_2'(t_0)| \le  2A_1 \, |\ti a_2(t_0)|^{\frac{1}{2}}.
  	\end{align*} 
  	It follows that, 
  	for $t_0 \in I_0'$ and $t \in I_{t_0}((6A_1)^{-1})$,  we have \eqref{long} and hence  \eqref{A.3}. 
  	%Then \eqref{A.3} together with \eqref{eq:incl} gives \eqref{A.2}.
	If $t \in I_{t_0}(A^{-1})$ then Lemma~\ref{taylor}, Lemma~\ref{lem:hyp}, and \eqref{A.3} imply \eqref{A.4}.
%  that	
%	    \begin{align*}%\label{eq:1}  
%    |\ti a_i^{(k)}(t) | \le C (n)\, A^{k}\, |\ti a_2(t_0)|^{\frac{i-k}{2}}.  
%  \end{align*}
  This ends the proof of Proposition \ref{lem:1}.  
\end{proof}

%-----------------------------------------------------------------------------------------------------------------------
\subsection{Bounds for \texorpdfstring{$p=n$}{p=n}}\label{subsectionbounds1}
%-----------------------------------------------------------------------------------------------------------------------

Let $\lambda (t)\in C^0 (I)$ be a root of $P_{\ti a}$ that is in Tschirnhausen form  and let  $I_0 \Subset I_1 \Subset I$.  By the 
inductive step \ref{ssec:indstep}, Proposition \ref{lem:1}, and Lemma \ref{lem:ext1} we have the following bounds 
\begin{align}\label{finalbounds}
  \on{Lip}_{I_{0}}(\la)  & \le  C(n) 
    \max\Big\{  \de^{-1}\|\ti a_2\|_{L^\infty( I_1)}^{\frac{1}{2}}, (\on{Lip}_{ I_1}(\ti a_2'))^{\frac 1 2},   
    \max_i \big\{M_i \|\ti a_2\|^{\frac {n-i} 2}_{L^\infty( I_1)} \big\} ^{\frac 1 n}   \Big\} \\
\notag	
		 &  \le C(n,I_0,I_1)\, \big(\max_i \|\ti a_i\|^{\frac 1 i}_{C^{n-1,1}(\overline I_1)}\big) 
		  \\
\notag	
		 &  \le C(n,I_0,I_1)\, \big(1+\max_i \|\ti a_i\|_{C^{n-1,1}(\overline I_1)}\big), 
\end{align}
where $\de$ is the distance between the endpoints of $I_0$ and those of $I_1$, and 
$M_i =   \on{Lip}_{ I_1}(\ti a_i ^{(n-1)})$.  
Then the bounds stated in Theorem \ref{main} follow from %the following
\begin{align*}%\label{finalbounds}
\max_i \|\ti a_i\|^{\frac 1 i}_{C^{n-1,1}(\overline I_1)}    
\le  C(n)\, \big(\max_i \| a_i\|^{\frac 1 i}_{C^{n-1,1}(\overline I_1)}\big)		 	  
\le C(n)\, \big(1+\max_i \| a_i\|_{C^{n-1,1}(\overline I_1)}\big) . 
\end{align*}
The first inequality follows from the (weighted) homogeneity of the formulas for $\tilde a_i$ in terms of $(a_1,\ldots,a_n)$.  
(The opposite inequality does not hold in general.  Adding a constant to all the roots of $P_{a}$ does not change the associated 
Tschirnhausen form $P_{\ti a}$ but changes the norm of the coefficients of $P_a$.)  

%-----------------------------------------------------------------------------------------------------------------------
\subsection{The case \texorpdfstring{$2\le p<n$}{2<= p<n}}\label{subsectionbounds2}
%-----------------------------------------------------------------------------------------------------------------------

To show that the roots are Lipschitz it suffices, using Lemma~\ref{split}, to split 
 $P_{\ti a}$ locally in factors of degree smaller than or equal to $p$ and apply the case $n=p$.   
 
  In order to have a uniform bound 
 we need to know that the multiplicities of roots are at most $p$ ``uniformly''.  For this 
 we order the roots of  $P_{\ti a}$ increasingly, $\la_1(t) \le \la_2(t) \le \cdots \le \la_n(t)$, and consider 
$$
\alpha (t) := \frac {|\la_n (t) - \la_1(t)|}  {\min_{i=1,...,n-p} |\la_{i+p}(t) -\la_{i}(t) |} ,  \quad \alpha_I := \sup _{t\in I} \alpha(t). 
$$
We note that the numerator ${|\la_n (t) - \la_1(t)|}  $ is of the same size as $|\ti a_2(t)|^{\frac 1 2}$,  
for $P_{\ti a}$ in Tschirnhausen form, since then $\lambda_1(t)$ and $\lambda_n(t)$ have opposite signs and
$$
n {|\la_n (t) - \la_1(t)|} \ge s_2(t)^{\frac 1 2}  = \sqrt 2 |\ti a_2(t)|^{\frac 1 2} \ge \frac  1 2  {|\la_n (t) - \la_1(t)|}.
$$

There are the following changes in the way we proceed.  
First in  the proof of Proposition~\ref{lem:1} we have to modify the formula for $A_2$ as follows 
 	\begin{align}\label{newA2}
		  A_2 :=  \max \Bigl \{ \max_{i\le p} \big\{M_i \|\ti a_2\|^{\frac {p-i} 2}_{L^\infty( I_1)} \big\} ^{\frac 1 p}, 
		   \max_{i> p} \big\{M_i \, m_2^{\frac {p-i} 2}   \big\} ^{\frac 1 p} \Bigr \}, 
	\end{align}
where $M_i =   \on{Lip}_{ I_1}(\ti a_i ^{(p-1)})$ and $m_2 = \min_{t\in \overline I_0} |\ti a_2(t)|$.

 In \eqref{A.4} of Assumption we may consider only  the derivatives of order $k\le p$.   
 Therefore the argument of the inductive step (proof of Lemma \ref{lem:induction}) changes as follows.  
 The first part of the proof of Lemma  \ref{lem:induction} does not change.  
 Then we need \eqref{B.4} for $i=k$ in order to apply Lemma \ref{shortcut}.  This is not available if $i>p$, which happens 
 if $\deg P_b >p$,  and then we have only 
 \[
    |{\tilde b_i}^{(p)}| \le C(n)\,  A^p\, |\ti a_2|^{\frac{i-p}{2}} \le C(n)\,  A_b ^p\, |\tilde b_2|^{\frac{i-p}{2}} ,
  \]
  where we may take $A_b = C(n)  |\ti a_2(t_1) / \tilde b_2(t_1)|^{\frac {n-p} {2p}} A $.  Then, by Lemma \ref{taylor}, 
  we conclude Lemma~\ref{lem:induction} with $A$ replaced by $A_b$.  This modification is no longer necessary 
  when $\deg P_b \le p$.  Thus during the induction process, say,  $P_{\ti a} \to P_{\ti b} \to \cdots \to P_{\ti d}\to P_{\ti e}$ 
    with $\deg P_e \le p$, for the intervals $ I_{t_0}(A^{-1} ) \supset I_ {t_1} (A_b^{-1}) \supset \cdots \supset I_ {t_s} (A_d^{-1}) $, 
  the constant $A$ is replaced by 
 \[   \tilde A = C(n) \alpha (t_s) ^{\frac {n-p} {p}} A \ge C(n) \Big( \Big|\frac{\ti a_2(t_s)}{\ti b_2(t_s)}\Big| 
    \cdot  \Big|\frac{\ti b_2(t_s)}{\ti c_2(t_s)}\Big|
    \cdots  \Big|\frac{\ti c_2(t_s)}{\ti d_2(t_s)}\Big|\Big)^{\frac {n-p} {2p}} A.  
  \]
  
Finally this gives the following bounds on the Lipschitz constant of each of the roots 
\begin{align} \label{finalbounds2}
  \begin{split}
  \MoveEqLeft   
  \on{Lip}_{I_{0}}(\la) %\le  
  \\    
  \le&\,      C(n)\,    \alpha_{I_1}^{\frac {n-p} {p}} 
    \max\Big\{  \de^{-1}\|\ti a_2\|_{L^\infty( I_1)}^{\frac{1}{2}}, (\on{Lip}_{ I_1}(\ti a_2'))^{\frac 1 2},   
    \max_{i\le p} \big\{M_i \|\ti a_2\|^{\frac {p-i} 2}_{L^\infty( I_1)} \big\} ^{\frac 1 p},  
    \max_{i> p} \big\{M_i \, m_2^{\frac {p-i} 2}  \big\} ^{\frac 1 p}  \Big\}  
  \\
  \le&\,   C(n, I_0, I_1)\,   \alpha_{I_1}^{\frac {n-p} {p}} \, \bigl (1+ m_2^{\frac{p-n}{2p}} \bigr ) \,  \big(1+\max_i \|\ti a_i\|_{C^{p-1,1}(\overline I_1)}\big) . 
  \end{split}  
\end{align}

This completes the proof of Theorem~\ref{main}. \qed

%-----------------------------------------------------------------------------------------------------------------------
\section{Proof of Theorem \ref{mainC1}} 
%-----------------------------------------------------------------------------------------------------------------------

%-----------------------------------------------------------------------------------------------------------------------
\subsection{\texorpdfstring{$\tC{p}{m}$}{pCm}-functions}
%-----------------------------------------------------------------------------------------------------------------------

In the proof of Theorem \ref{mainC1} we shall need a result for functions  defined near $0 \in\R$ that become $C^m$ when 
multiplied with the monomial $t^p$. 

\begin{definition}
  Let $p,m \in \N$ with $p \le m$.
A continuous complex valued function $f$ defined near $0 \in \R$ is called a \emph{$\tC{p}{m}$-function} 
if $t \mapsto t^p f(t)$ belongs to $C^m$. 
\end{definition}

Let $I \subseteq \R$ be an open interval containing $0$. 
Then $f : I \to \C$ is $\tC{p}{m}$ if and only if 
it has the following properties, cf.\ \cite[4.1]{Spallek72}, \cite[Satz~3]{Reichard79}, or \cite[Thm~4]{Reichard80}:
\begin{itemize}
\item $f \in C^{m-p}(I)$,
\item $f|_{I\setminus \{0\}} \in C^m(I\setminus \{0\})$,
\item $\lim_{t \to 0} t^k f^{(m-p+k)}(t)$ exists as a finite number for all $0 \le k \le p$. 
\end{itemize}

\begin{proposition} \label{prop:pCm}
If $g=(g_1,\ldots,g_n)$ is $\tC{p}{m}$ and $F$ is $C^m$ near $g(0)\in \C^n$, then $F \o g$ is $\tC{p}{m}$.
\end{proposition}

\begin{proof}
Cf.\ \cite[Thm~9]{Reichard80} or \cite[Prop~3.2]{RainerFin}. Clearly $g$ and $F \o g$ are $C^{m-p}$ near $0$ and $C^m$ off $0$. 
By Fa\`a di Bruno's formula \cite{FaadiBruno1855}, for $1 \le k \le p$ and $t\ne0$,
\begin{align*} 
    \frac{t^k (F \o g)^{(m-p+k)}(t)}{(m-p+k)!} &= \sum_{\ell\ge 1}  \sum_{\al \in A}
    \frac{t^{k-|\be|}}{\ell!}  d^\ell F(g(t)) 
    \Big( 
    \frac{t^{\be_1} g^{(\al_1)}(t)}{\al_1!},\dots,
    \frac{t^{\be_\ell} g^{(\al_\ell)}(t)}{\al_\ell!}\Big) \\
    A &:= \{\al\in \N_{>0}^\ell : \al_1+\dots+\al_\ell =m-p+k\} \\
    \be_i &:= \max\{\al_i - m + p,0\}, \quad |\be| = \be_1+\dots+\be_\ell \le k,
\end{align*}
whose limit as $t\to 0$ exists as a finite number by assumption.
\end{proof}

Let us prove Theorem \ref{mainC1}.
We suppose that $P_a$ is in Tschirnhausen form $P_a=P_{\ti a}$.  It suffices to consider the case $n=p$.  
We show that every $t_0\in I$ has a neighborhood in $I$ on which (1) and (2) (of Theorem \ref{mainC1}) hold.  
 If $\ti a_2(t_0) \ne 0$ then $P_{\ti a}$ splits on a 
neighborhood of $t_0$ and we may proceed by induction on $\deg P_a$.  
If $\ti a_2 (t_0)=0$ then $\ti a'_2 (t_0)=0$ and we distinguish two cases 
\begin{itemize}
\item
\emph{Case (i):} 
$\ti a_2(t_0) = \ti a'_2(t_0) = \ti a''_2(t_0)=0$.
\item 
\emph{Case (ii):} 
$\ti a_2(t_0) = \ti a'_2(t_0)=0$ and $ \ti a''_2(t_0)\ne 0$.
\end{itemize}
To simplify the notation we suppose $t_0=0$.  Fix a continuous root $\lambda(t)$ defined in a 
neighborhood of $0$. 

%-----------------------------------------------------------------------------------------------------------------------
\subsection{Proof of (1)} \label{Proofof1}
%-----------------------------------------------------------------------------------------------------------------------

In Case (i), $\lambda (t) = o(t)$ and hence $\lambda$ is differentiable at $0$ and 
$\lambda' (0) =0$.  
In Case (ii),  $\ti a_2(t) \sim t^2$  and hence $\ti a_i(t) = O(t^i)$.  Therefore, %in a neighborhood of $0$,  
\[
	\underline a (t) := \big(t^ {-2} \ti a_2(t), t^{-3 } \ti a_3 (t),\ldots, t^{-n } \ti a_n (t) \big) : 
	I_1 \to  \R^{n-1}_{\underline a} 
\]
defined on a neighborhood $I_1$ of $0$ is continuous. 
By  Lemma~\ref{split}, $P_{\underline a}$ splits.  
The splitting 
$P_{\underline a} = P_{\underline b} P_{\underline c}$  induces a splitting 
$P_{\ti a} = P_b P_c$, 
where the $b_i$ are given by 
 \begin{equation} \label{eq:3t}
   b_i = t^{{i}} \ps_i \big(t^{-2} \ti a_2, \ldots, t^{-n} \ti a_n\big), 
    \quad i = 1,\ldots,\deg P_b;
  \end{equation}
  and similar formulas hold for $\tilde b_i$.   Then $b_i$ and $\tilde b_i$ are of class $C^i$ at $0$, 
  by Proposition \ref{prop:pCm}, 
  and of class 
  $C^n$ in the complement of $0$.   
  Moreover we may choose the splitting such that $\lambda (t) $ for $t\ge 0$ 
   is a root of $P_b$,  and all the roots of $P_{\underline b(0)}$ are equal.  The latter gives  
  $$
  \tilde b_2(0) =   \tilde b'_2(0) =   \tilde b''_2(0) = 0. 
  $$
   Thus,  $\lambda (t)$ can be expressed as in \eqref{sumforaroot} with  
  $b_1$ of class $C^1$ and $\mu_j$ differentiable at $0$ ($\mu'_j(0) =0$). 
  This finishes the proof of (1).  
  
%-----------------------------------------------------------------------------------------------------------------------  
\subsection{Proof of (2)} 
%-----------------------------------------------------------------------------------------------------------------------

  This is the heart of the proof.  In Case (i) the continuity of the one-sided derivatives 
  at $0$ follows from \eqref{finalbounds} and the following lemma.  

\begin{lemma}\label{flatcaselemma}
Suppose Case (i) holds.  
Then for any $\varepsilon>0$ there is $\delta >0$ such that for  
$I_0 = (-\delta, \delta)$ and $I_1 = (-2\delta, 2\delta)$  and $A_0$ defined by \eqref{formulaforA} we have 
 $A_0 \le \varepsilon$.  
\end{lemma}

\begin{proof}
This follows immediately from the formula \eqref{formulaforA}.  
\end{proof}

To show the continuity in Case (ii) we need a similar result for $P_{\tilde b}$.  
 
 \begin{lemma}
Suppose Case (ii) holds.  
Then, under the assumptions of Subsection \ref{Proofof1},   
for any $\varepsilon>0$ there is a neighborhood $I_\varepsilon$ of $0$ in $I$ such that 
for every  $t_0\in I_\varepsilon \setminus \{0\}$ the conditions \eqref{A.2}--\eqref{A.4} are satisfied 
for $P_{\tilde b}$ with $A\le \varepsilon$.  
\end{lemma}

\begin{proof}
Since $P_{\tilde b}$ is not necessarily of class 
$C^{\deg P_b}$ we cannot use directly Lemma \ref{flatcaselemma} and the induction on $\deg P_a$.  
But the proof is similar and we sketch it below.  

Let $I_1 =I_ \delta = (-\delta, \delta)$ and $I_0 = (- \frac \delta 2, \frac \delta 2 )$.  Since $\tilde b''_2(0) = 0 $ and $\tilde b_2(t)$ is 
of class $C^2$,  the constant $A_1$ of \eqref{A1A2}  for $\tilde b$ can be made arbitrarily small, provided $\delta$ is chosen sufficiently 
small.  This is what we  need to get \eqref{A.2}--\eqref{A.3} with arbitrarily small $A$. 
 
By Lemma \ref {lem:hyp},  $\tilde b_i^{(k)} (0)=0$ for $i=2, ..., \deg P_b$, $k=0, ... ,i$.  Fix $A>0$.  Since every $\tilde b_i$ is 
of class $C^i$, there is a neighborhood $I_\delta$ in which \eqref{A.4} holds for  $i=2, ..., n$, $k=i$,  and then, 
by Lemma \ref{shortcut}, in a smaller neighborhood,  also for $i=2, ..., n$, $k\le i$.  

Finally, given $A>0$ we show  \eqref{A.4} for $i<k\le n$ and $\delta$ sufficiently small.  
      Let $\hat A$ denote the constant $A$ for which \eqref{A.2}--\eqref{A.4} holds for $P_a$.  By 
     \eqref{eq:b1}, 
    \begin{equation*}% \label{eq:b1}
    |\ti b_i ^{(k)}(t) | \le C(n)  \hat A ^k\, |\ti a_2(t) |^{\frac{i-k}{2}} \le C(n) \hat A^k \, \varphi (t)  |\tilde b_2 (t) |^{\frac{i-k}{2}},
  \end{equation*}
  which gives the required result since,  for $k>i$,   $\varphi (t) =  | \tilde b_2 (t) / \ti  a_2(t) |^{\frac{k-i}{2}} = o(1)$. 
%
%Finally, given $A>0$ we show  \eqref{A.4} for $i<k\le n$ and $\delta$ sufficiently small.  
%Let $M_i =  \|\ti a_i^{(n)}\|_{L^\infty (I)}$ and choose $\delta$ so that for $t\in I_\delta$ and $i=1,... ,\deg P_b$,  
%\begin{align}
%|\ti a_2(t)|^{\frac {n-i}2}(2\delta)^n \le  M_i ^{-1}.  
%\end{align}
%By Lemma \ref{taylor}, there is  $C=C(n)>1$ such that for any  $t_0\in I'_0$  and $t\in I_{t_0}(C^{-1} \delta)$
%    \begin{align*}  
%    |\ti  a_i^{(k)}(t) |& 
%     \le C(n)\, (2\delta)^{-k} |\ti  a_2(t_0)|^{-\frac k 2}  \bigl(|\ti a_2(t_0)|^{\frac i 2} +  M_i (2\delta)^n |\ti a_2(t_0)|^{\frac n 2}    \bigr)  
%     %\\& 
%     \le  C(n)\, (2\delta)^{-k} |\ti a_2(t_0)|^{\frac {i-k} 2} .  
%  \end{align*}
%  Using an analog of \eqref{derivatives} we get, for a constant $C_1$ that is independent of $t$,   
%  \begin{equation*}% \label{eq:b1}
%    |\ti b_i ^{(k)}(t) | \le C_1 \, |\ti a_2(t) |^{\frac{i-k}{2}} \le C_1 \, \varphi (t)  |\tilde b_2 (t) |^{\frac{i-k}{2}},
%  \end{equation*}
%  that gives the required result since,  for $k>i$,   $\varphi (t) =  | \tilde b_2 (t) / \ti  a_2(t) |^{\frac{k-i}{2}} = o(1)$.  
%  That gives the result.  
\end{proof} 

%-----------------------------------------------------------------------------------------------------------------------
\subsection{Proof of (3)}
%-----------------------------------------------------------------------------------------------------------------------

 We proceed by induction on $n$. The case $n=1$ is obvious. So assume $n>1$. 
  Set $F = \{t \in I : \ti  a_2(t) =\ti  a''_2(t) = 0\}$. Its  
  complement is a countable union of disjoint open intervals, $I \setminus F = \bigcup_k I_k$.
  At each $t_0 \in I \setminus F$ the polynomial $P_{\ti a}$ splits, and, by the induction hypothesis, 
  there exists a local differentiable system of the roots of $P_{\ti a}$ near $t_0$. We may infer that 
  there exists a differentiable system on each interval $I_k$. For, if the (say) right 
  endpoint $t_1$ of the domain $I_{\la}$ of  $\la = (\la_j)_{j=1}^n$ belongs to $I_k$, there exists a local 
  system $\mu = (\mu_j)_{j=1}^n$ with $t_1 \in I_{\mu}$. 
  We may choose $t_2 \in I_{\la} \cap I_{\mu}$ and extend $(\la_j)_j$ by $(\mu_{\si(j)})_j$ on the 
  right of $t_2$ beyond $t_1$, where $\si$ is a suitable permutation.  
  Extending by $0$ on $F$ yields a differentiable system $(\la_j)_j$ of the roots on $I$ (the derivatives vanish on $F$).  

%-----------------------------------------------------------------------------------------------------------------------
\subsection{Proof of (4)}
%-----------------------------------------------------------------------------------------------------------------------

  It follows immediately from (2).

\def\cprime{$'$}
\providecommand{\bysame}{\leavevmode\hbox to3em{\hrulefill}\thinspace}
\providecommand{\MR}{\relax\ifhmode\unskip\space\fi MR }
% \MRhref is called by the amsart/book/proc definition of \MR.
\providecommand{\MRhref}[2]{%
  \href{http://www.ams.org/mathscinet-getitem?mr=#1}{#2}
}
\providecommand{\href}[2]{#2}

%\bibliography{../../references/biblio}

\begin{thebibliography}{10}

\bibitem{AKLM98}
D.~Alekseevsky, A.~Kriegl, M.~Losik, and P.~W. Michor, \emph{Choosing roots of
  polynomials smoothly}, Israel J. Math. \textbf{105} (1998), 203--233.

\bibitem{BM90}
E.~Bierstone and P.~D. Milman, \emph{Arc-analytic functions}, Invent. Math.
  \textbf{101} (1990), no.~2, 411--424.

\bibitem{BBCP06}
J.-M. Bony, F.~Broglia, F.~Colombini, and L.~Pernazza, \emph{Nonnegative
  functions as squares or sums of squares}, J. Funct. Anal. \textbf{232}
  (2006), no.~1, 137--147.

\bibitem{BonyColombiniPernazza06}
J.-M. Bony, F.~Colombini, and L.~Pernazza, \emph{On the differentiability class
  of the admissible square roots of regular nonnegative functions}, Phase space
  analysis of partial differential equations, Progr. Nonlinear Differential
  Equations Appl., vol.~69, Birkh\"auser Boston, Boston, MA, 2006, pp.~45--53.

\bibitem{BonyColombiniPernazza10}
\bysame, \emph{On square roots of class {$C^m$} of nonnegative functions of one
  variable}, Ann. Sc. Norm. Super. Pisa Cl. Sci. (5) \textbf{9} (2010), no.~3,
  635--644. %\MR{2722658 (2011j:26004)}

\bibitem{Bronshtein79}
M.~D. Bronshtein, \emph{Smoothness of roots of polynomials depending on
  parameters}, Sibirsk. Mat. Zh. \textbf{20} (1979), no.~3, 493--501, 690,
  English transl. in Siberian Math. J. \textbf{20} (1980), 347--352.

\bibitem{ColombiniOrruPernazza12}
F.~Colombini, N.~Orr{{\`u}}, and L.~Pernazza, \emph{On the regularity of the
  roots of hyperbolic polynomials}, Israel J. Math. \textbf{191} (2012),
  923--944. %\MR{2970893}

\bibitem{FaadiBruno1855}
C.~F. Fa\`a~di Bruno, \emph{Note {s}ur {u}ne {n}ouvelle {f}ormule {d}u {c}alcul
  {d}iff\'erentielle}, Quart. J. Math. \textbf{1} (1855), 359--360.

\bibitem{Glaeser63R}
G.~Glaeser, \emph{Racine carr\'ee d'une fonction diff\'erentiable}, Ann. Inst.
  Fourier (Grenoble) \textbf{13} (1963), no.~2, 203--210.

\bibitem{KLM04}
A.~Kriegl, M.~Losik, and P.~W. Michor, \emph{Choosing roots of polynomials
  smoothly. {II}}, Israel J. Math. \textbf{139} (2004), 183--188.

\bibitem{KurdykaPaunescu08}
K.~Kurdyka and L.~Paunescu, \emph{Hyperbolic polynomials and multiparameter
  real-analytic perturbation theory}, Duke Math. J. \textbf{141} (2008), no.~1,
  123--149.

\bibitem{Mandai85}
T.~Mandai, \emph{Smoothness of roots of hyperbolic polynomials with respect to
  one-dimensional parameter}, Bull. Fac. Gen. Ed. Gifu Univ. (1985), no.~21,
  115--118.

\bibitem{Procesi78}
C.~Procesi, \emph{Positive symmetric functions}, Adv. in Math. \textbf{29}
  (1978), no.~2, 219--225.

\bibitem{RS02}
Q.~I. Rahman and G.~Schmeisser, \emph{Analytic theory of polynomials}, London
  Mathematical Society Monographs. New Series, vol.~26, The Clarendon Press
  Oxford University Press, Oxford, 2002.

\bibitem{RainerOmin}
A.~Rainer, Smooth roots of hyperbolic polynomials with definable
  coefficients, \emph{Israel J. Math.} \textbf{184} (2011), 157--182. 
  %\MR{2823973 (2012m:26013)}

\bibitem{RainerN}
\bysame, Perturbation theory for normal operators, \emph{Trans. Amer. Math.
  Soc.} \textbf{365} (2013), no.~10, 5545--5577. %\MR{3074382}

\bibitem{RainerFin}
\bysame, Differentiable roots, eigenvalues, and eigenvectors, \emph{Israel
  J.\ Math.}, \textbf{201} (2014), No. 1, 99--122. 


\bibitem{Reichard79}
K.~Reichard, \emph{Algebraische {B}eschreibung der {A}bleitung bei {$q$}-mal
  stetig-differenzierbaren {F}unktionen}, Compositio Math. \textbf{38} (1979),
  no.~3, 369--379.

\bibitem{Reichard80}
\bysame, \emph{Roots of differentiable functions of one real variable}, J.
  Math. Anal. Appl. \textbf{74} (1980), no.~2, 441--445.

\bibitem{Rellich37I}
F.~Rellich, \emph{St\"orungstheorie der {S}pektralzerlegung}, Math. Ann.
  \textbf{113} (1937), no.~1, 600--619.

\bibitem{Spallek72}
K.~Spallek, \emph{Abgeschlossene {G}arben differenzierbarer {F}unktionen},
  Manuscripta Math. \textbf{6} (1972), 147--175. %\MR{0300301 (45 \#9347)}

\bibitem{Tarama06}
S.~Tarama, \emph{Note on the {B}ronshtein theorem concerning hyperbolic
  polynomials}, Sci. Math. Jpn. \textbf{63} (2006), no.~2, 247--285.

\bibitem{Wakabayashi86}
S.~Wakabayashi, \emph{Remarks on hyperbolic polynomials}, Tsukuba J. Math.
  \textbf{10} (1986), no.~1, 17--28.

\end{thebibliography}
%\bibliographystyle{amsplain}

\end{document}